\documentclass[12pt]{amsart}

\usepackage{amssymb, amsmath, amsthm, amsfonts}
\usepackage{mathrsfs,comment}
\usepackage{graphicx}
\usepackage{placeins}
\usepackage{todonotes}
\usepackage{rotating}
\usepackage{tikz}
\usepackage{float}
\usepackage{subfigure}
\usepackage{hvfloat}
\usepackage{caption}
\usepackage{pdflscape}
\usepackage{hyperref}  
\usepackage{url}
\usepackage[all,arc,2cell]{xy}
\UseAllTwocells
\usepackage{enumerate}
\usepackage{chngcntr}
 \usepackage{lineno}
 \usepackage{blindtext}
\usepackage{verbatim}
\usepackage{soul}
\usepackage{sseq}
\usepackage{mathtools}
\usepackage{times}
\usepackage{tikz-cd}
\usepackage{pgfplots}
\usepackage{pgfplotstable}

\usepackage[normalem]{ulem}
\newcommand{\stkout}[1]{\ifmmode\text{\sout{\ensuremath{#1}}}\else\sout{#1}\fi}

    \hypersetup{%
    bookmarksnumbered=true,%
    bookmarks=true,%
    colorlinks=true,%
    linkcolor=blue,%
    citecolor=blue,%
    filecolor=blue,%
    menucolor=blue,%
    pagecolor=blue,%
    urlcolor=blue,%
    pdfnewwindow=true,%
    pdfstartview=FitBH}

\def\@url#1{{\tt\def~{\lower3.5pt\hbox{\char'176}}\def\_{\char'137}#1}}

\let\fullref\autoref
\def\makeautorefname#1#2{\expandafter\def\csname#1autorefname\endcsname{#2}}
\makeautorefname{equation}{Equation}%
\makeautorefname{footnote}{footnote}%
\makeautorefname{item}{item}%
\makeautorefname{figure}{Figure}%
\makeautorefname{table}{Table}%
\makeautorefname{part}{Part}%
\makeautorefname{appendix}{Appendix}%
\makeautorefname{chapter}{Chapter}%
\makeautorefname{section}{Section}%
\makeautorefname{subsection}{Section}%
\makeautorefname{subsubsection}{Section}%
\makeautorefname{paragraph}{Paragraph}%
\makeautorefname{subparagraph}{Paragraph}%
\makeautorefname{theorem}{Theorem}%
\makeautorefname{theo}{Theorem}%
\makeautorefname{thm}{Theorem}%
\makeautorefname{addendum}{Addendum}%
\makeautorefname{addend}{Addendum}%
\makeautorefname{add}{Addendum}%
\makeautorefname{maintheorem}{Main theorem}%
\makeautorefname{mainthm}{Main theorem}%
\makeautorefname{corollary}{Corollary}%
\makeautorefname{claim}{Claim}%
\makeautorefname{corol}{Corollary}%
\makeautorefname{coro}{Corollary}%
\makeautorefname{cor}{Corollary}%
\makeautorefname{lemma}{Lemma}%
\makeautorefname{lemm}{Lemma}%
\makeautorefname{lem}{Lemma}%
\makeautorefname{sublemma}{Sublemma}%
\makeautorefname{sublem}{Sublemma}%
\makeautorefname{subl}{Sublemma}%
\makeautorefname{proposition}{Proposition}%
\makeautorefname{proposit}{Proposition}%
\makeautorefname{propos}{Proposition}%
\makeautorefname{propo}{Proposition}%
\makeautorefname{prop}{Proposition}%
\makeautorefname{property}{Property}
\makeautorefname{proper}{Property}
\makeautorefname{scholium}{Scholium}%
\makeautorefname{step}{Step}%
\makeautorefname{conjecture}{Conjecture}%
\makeautorefname{conject}{Conjecture}%
\makeautorefname{conj}{Conjecture}%
\makeautorefname{question}{Question}
\makeautorefname{questn}{Question}
\makeautorefname{quest}{Question}
\makeautorefname{ques}{Question}
\makeautorefname{qn}{Question}
\makeautorefname{definition}{Definition}%
\makeautorefname{defin}{Definition}%
\makeautorefname{defi}{Definition}%
\makeautorefname{def}{Definition}%
\makeautorefname{defn}{Definition}%
\makeautorefname{dfn}{Definition}%
\makeautorefname{notation}{Notation}
\makeautorefname{nota}{Notation}
\makeautorefname{notn}{Notation}
\makeautorefname{remark}{Remark}%
\makeautorefname{rema}{Remark}%
\makeautorefname{rem}{Remark}%
\makeautorefname{rmk}{Remark}%
\makeautorefname{rk}{Remark}%
\makeautorefname{remarks}{Remarks}%
\makeautorefname{rems}{Remarks}%
\makeautorefname{rmks}{Remarks}%
\makeautorefname{rks}{Remarks}%
\makeautorefname{example}{Example}%
\makeautorefname{examp}{Example}%
\makeautorefname{exmp}{Example}%
\makeautorefname{exmps}{Examples}%
\makeautorefname{exam}{Example}%
\makeautorefname{exa}{Example}%
\makeautorefname{algorithm}{Algorith}%
\makeautorefname{algo}{Algorith}%
\makeautorefname{alg}{Algorith}%
\makeautorefname{axiom}{Axiom}%
\makeautorefname{axi}{Axiom}%
\makeautorefname{ax}{Axiom}%
\makeautorefname{case}{Case}%
\makeautorefname{claim}{Claim}%
\makeautorefname{clm}{Claim}%
\makeautorefname{assumption}{Assumption}%
\makeautorefname{assumpt}{Assumption}%
\makeautorefname{conclusion}{Conclusion}%
\makeautorefname{concl}{Conclusion}%
\makeautorefname{conc}{Conclusion}%
\makeautorefname{condition}{Condition}%
\makeautorefname{condit}{Condition}%
\makeautorefname{cond}{Condition}%
\makeautorefname{construction}{Construction}%
\makeautorefname{construct}{Construction}%
\makeautorefname{const}{Construction}%
\makeautorefname{cons}{Construction}%
\makeautorefname{criterion}{Criterion}%
\makeautorefname{criter}{Criterion}%
\makeautorefname{crit}{Criterion}%
\makeautorefname{exercise}{Exercise}%
\makeautorefname{exer}{Exercise}%
\makeautorefname{exe}{Exercise}%
\makeautorefname{problem}{Problem}%
\makeautorefname{problm}{Problem}%
\makeautorefname{probm}{Problem}%
\makeautorefname{prob}{Problem}%
\makeautorefname{solution}{Solution}%
\makeautorefname{soln}{Solution}%
\makeautorefname{sol}{Solution}%
\makeautorefname{summary}{Summary}%
\makeautorefname{summ}{Summary}%
\makeautorefname{sum}{Summary}%
\makeautorefname{operation}{Operation}%
\makeautorefname{oper}{Operation}%
\makeautorefname{observation}{Observation}%
\makeautorefname{observn}{Observation}%
\makeautorefname{obser}{Observation}%
\makeautorefname{obs}{Observation}%
\makeautorefname{ob}{Observation}%
\makeautorefname{convention}{Convention}%
\makeautorefname{convent}{Convention}%
\makeautorefname{conv}{Convention}%
\makeautorefname{cvn}{Convention}%
\makeautorefname{warning}{Warning}%
\makeautorefname{warn}{Warning}%
\makeautorefname{note}{Note}%
\makeautorefname{fact}{Fact}%
\makeautorefname{thmbig}{Theorem}%
\makeautorefname{conjbig}{Conjecture}%

  \makeatletter
                   \let\c@lemma\c@theorem
                  \makeatother

\newtheorem{thm}{Theorem}[section]
\newtheorem{theorem}{Theorem}[section]
\newtheorem{cor}{Corollary}[section]
\newtheorem{prop}{Proposition}[section]
\newtheorem{lem}{Lemma}[section]
\newtheorem{conj}{Conjecture}[section]

\theoremstyle{definition}

\newtheorem{rem}{Remark}[section]

\makeatletter
\let\c@lem=\c@thm
\let\c@theorem=\c@thm
\let\c@notation=\c@thm
\let\c@cor=\c@thm
\let\c@prop=\c@thm
\let\c@lem=\c@thm
\let\c@defn=\c@thm
\let\c@exmps=\c@thm
\let\c@rem=\c@thm
\let\c@warn=\c@thm
\let\c@claim=\c@thm
\let\c@conj=\c@thm
\let\c@quest=\c@thm
\makeatother

\numberwithin{equation}{section}

\def\quickop#1{\expandafter\newcommand\csname #1\endcsname{\operatorname{#1}}}
\quickop{Hom} \quickop{End} \quickop{Aut} \quickop{Tel} \quickop{Mic} \quickop{map}
\quickop{Ext} \quickop{Tor} \quickop{Cotor} \quickop{Id} \quickop{Coker} \quickop{Ker}
\quickop{Lim} \quickop{Colim} \quickop{Holim} \quickop{Hocolim}
\quickop{id} \quickop{tel} \quickop{mic} \quickop{coker}
\quickop{colim} \quickop{holim} \quickop{hocolim} \quickop{im} \quickop{Syl} \quickop{Ind}

\numberwithin{equation}{subsection}

\newcommand{\F}{\mathbb{F}}

\newcommand{\G}{\mathbb{G}}
\newcommand{\W}{\mathbb{W}}
\newcommand{\Q}{\mathbb{Q}}

\newcommand{\FF}{\mathbb{F}}

\DeclareFontFamily{OMS}{rsfs}{\skewchar\font'60}
\DeclareFontShape{OMS}{rsfs}{m}{n}{<-5>rsfs5 <5-7>rsfs7 <7->rsfs10 }{}
\DeclareSymbolFont{rsfs}{OMS}{rsfs}{m}{n}
\DeclareSymbolFontAlphabet{\scr}{rsfs}

\newcommand{\tr}{\alpha}

\def\makeop#1{\expandafter\def\csname #1\endcsname{\mathop{\mathrm{#1}}\nolimits}}

\makeop{Gal}
\makeop{id}
\makeop{Mod}
\makeop{Tot}
\makeop{gr}
\makeop{Out}
\makeop{Hom}
\makeop{Ext}
\makeop{End}
\makeop{Aut}
\makeop{Tor}
\makeop{ev}
\makeop{Sym}
\def\FF{\mathbb{F}}

\def\GG{\mathbb{G}}
\def\SS{\mathbb{S}}
\def\WW{{{\mathbb{W}}}}

\def\Z{{{\mathbb{Z}}}}

\newcommand{\LR}{R}

\usepackage{fullpage}

\title[Orbits for the Lubin--Tate Ring]{Computations of Orbits for the Lubin--Tate Ring}
\date{\today}

\author[A. Beaudry]{Agn\`es Beaudry}
\address{Department of Mathematics\\ University of Colorado Boulder \\ \newline Campus Box 395 \\ Boulder \\ Colorado \\ 80309 \\ USA}

\author[N. Downey]{Naiche Downey}

\author[C. McCranie]{Connor McCranie}

\author[L. Meszar]{Luke Meszar}

\author[A. Riddle]{Andy Riddle}

\author[P. Rock]{Peter Rock}

\thanks{This material is based on work supported by the CU Boulder Department of Mathematics in the context of its internal Research For Undergraduates program. This material is also based upon work supported by the National Science Foundation under Grant No.~DMS-1725563.}

\begin{document}
\maketitle
\begin{abstract}
We take a direct approach to computing the orbits for the action of the automorphism group $\mathbb{G}_2$ of the Honda formal group law of height $2$ on the associated Lubin--Tate rings $R_2$. We prove that $(R_2/p)_{\mathbb{G}_2}  \cong \mathbb{F}_p$. The result is new for $p=2$ and $p=3$. For primes $p\geq 5$, the result is a consequence of computations of Shimomura and Yabe and has been reproduced by Kohlhaase using different methods.
\end{abstract}
\setcounter{tocdepth}{2}
\tableofcontents

\section{Introduction}
In this paper, we consider a direct approach to computing orbits for the action of the automorphism group of the Honda formal group law of height $2$ on the reduction modulo $(p)$ of the associated Lubin--Tate ring. The results are new for $p=2$ and $p=3$ and they follow from the work of Shimomura and Yabe \cite{shimyab} if $p\geq 5$, also reproduced by Kohlhaase in \cite{Kohlhaase}. We also use this as an opportunity to highlight some of the results on the action of the automorphism group which appeared in French in the doctoral thesis of Lader \cite{lader}. See \fullref{sec:lader}.

These results are meant to lend weight to a conjecture, which for lack of a better name we will call the Chromatic Vanishing Conjecture. This conjecture plays a key role in the analysis of Hopkins' Chromatic Splitting Conjecture (as stated by Hovey in \cite{cschov}) at the prime $p=3$ in \cite{GoerssSplit} and at the prime $p=2$ in \cite{BGH}. See \fullref{rem:cscexp} below. The importance this statement plays at height $n=2$ was originally highlighted to the last author by Hans-Werner Henn. To state it, consider the Honda formal group law of height $n$ over $\F_{p^n}$. The associated Lubin--Tate ring $R_n$ satisfies 
$\LR_n \cong \WW[\![u_1, \ldots, u_{n-1}]\!]$
where $\WW$ are the Witt vectors on $\FF_{p^n}$. Let $\mathbb{H}_n$ be the Honda formal group law of height $n$ and $\mathbb{S}_n $ be the group of automorphisms of $\mathbb{H}_n$ over $\F_{p^n}$. Since $\mathbb{H}_n$ has coefficients in $\F_p$, the Galois group $\Gal(\F_{p^n}/\F_p)$ acts on $\mathbb{S}_n$. We  
let
$\GG_n$ be the extension of $\mathbb{S}_n$ by the Galois group. 
\begin{conj}[Chromatic Vanishing Conjecture]\label{conj:main}
Let $\W \to  \LR_n$ and $\FF_{p^n} \to \LR_n/p$ be the natural maps.
\begin{enumerate}[(1)]
\item (Integral) The continuous cohomology and homology of $\LR_n/\W$ vanish in all degrees so that
\begin{align*}
H^*(\G_n, \LR_n ) &\cong  H^*(\G_n, \W)  & H_*(\G_n, \LR_n) &\cong  H_*(\G_n, \W).
\end{align*}
\item (Reduced) The continuous cohomology and homology of $(\LR_n/p)/\F_{p^n}$ vanish in all degrees so that
\begin{align*}
H^*(\G_n, \LR_n/p ) &\cong  H^*(\G_n, \F_{p^n})  & H_*(\G_n, \LR_n/p) &\cong  H_*(\G_n, \F_{p^n})
\end{align*}
\end{enumerate}
\end{conj}

When $p\gg n$, the groups $\G_n$ are oriented Poincar\'e duality groups and the statements for cohomology and homology are equivalent. Further, the reduced conjectures imply their integral versions. Indeed, using the five lemma, (2) implies the vanishing of the continuous cohomology and homology with coefficients in $(\LR_n/p^k)/(\W/p^k)$ for all $k\geq 1$. A ${\lim}^1$ exact sequence then gives the desired implication.

The conjecture is a tautology at height $n=1$. At height $n=2$, the statements about cohomology are known to hold for all primes. They are due to Shimomura--Yabe if $p\geq 5$ \cite{shimyab}, to Henn--Karamanov--Mahowald and Goerss--Henn--Mahowald--Rezk for $p=3$  \cite{hkm, GoerssSplit} and to Beaudry--Goerss--Henn for $p=2$ \cite{Paper2, BGH}. Kohlhaase has reproduced the results for $p\geq 5$ in \cite[Theorem 3.20]{Kohlhaase} using different methods. For $p\geq 5$, Poincar\'e duality then gives the homological results. Finally, that $H^0(\GG_n, R_n) \cong H^0(\GG_n, \W) \cong \Z_p$ at all heights and primes is a folklore result of Hopkins. See \cite[Lemma 1.33]{BobkovaGoerss}.

For $p=2$ and $p=3$, similar methods to those used to prove the cohomological results should give a proof of the conjecture for homology. As in the cohomological cases, this would probably be a tedious computation. However, in this paper, we prove the homological result modulo $(p)$ in degree zero via a direct argument for all primes, including $p=2$ and $p=3$. Our main theorem is:
\begin{thm}
Let $p$ be any prime. The natural map $ \F_{p^2} \to R_2/p$ induces an isomorphism
\[H_0(\GG_2, R_2/p) \cong H_0(\GG_2, \F_{p^2}).\]
\end{thm}

\begin{rem}\label{rem:cscexp}
We briefly explain the relationship of \fullref{conj:main} with the Chromatic Splitting Conjecture (CSC) as discussed in Section 4 of \cite{cschov}. Let $K(n)$ be the Morava $K$-theory spectrum and $E_n =E( \F_{p^n}, \mathbb{H}_n)$ be the Lubin-Tate spectrum, so that $(E_n)_0 \cong R_n$. By the Goerss--Hopkins--Miller Theorem \cite{goersshopkins}, the group $\GG_n$ acts on $E_n$ by maps of $\mathcal{E}_{\infty}$ ring spectra and a well-known result of Devinatz and Hopkins states that $L_{K(n)}S^0 \simeq E_n^{h\GG_n} $ \cite{DH}. Further, the $K(n)$-local $E_n$-based Adams-Novikov Spectral Sequence can be identified with the homotopy fixed point spectral sequence
\[E_2^{s,t} = H^s(\GG_n, (E_n)_t) \Longrightarrow \pi_{t-s}E_n^{h\GG_n} \cong \pi_{t-s}L_{K(n)}S^0 . \]

The CSC predicts that the chromatic reassembly process is governed by elements of $\pi_{*}L_{K(n)}S^0$ which are detected in $E_2^{*,0} \cong H^*(\GG_n, R_n)$ by classes in the image of the map from $H^*(\GG_n, \W)$. Based on a computation of Lazard and Morava \cite[Remark 2.2.5]{morava_noe}, the cohomological version of \fullref{conj:main} would immediately imply that the CSC holds rationally. Integrally, it would at the very least imply that the reassembly classes are present on the $E_2$--page. At large primes where the spectral sequence collapses, these classes would then exist in homotopy. Proving the cohomological version of \fullref{conj:main} is among the hardest computations in both \cite{GoerssSplit} and \cite{BGH}.
\end{rem}

At this time, a computational proof of the Chromatic Vanishing Conjecture at higher heights seems out of reach. One could hope for a computational proof in homological degree zero at general heights. However, the precision of the information on the action of $\GG_2$ needed to carry out our direct argument suggests that even in this case, a computational proof may not be feasible. Further, if it is true in general, it should not be a computational accident and there ought to be a compelling conceptual explanation.

\ \\
\noindent
\textbf{Organization of the paper.} In \fullref{sec:res}, we give the proof of the main result. In \fullref{sec:lader}, we review the formulas for the action of $\GG_2$ needed for the computations. 

\ \\
\noindent
\textbf{Acknowledgements.} We thank some of the usual suspects for useful conversations: Tobias Barthel, Mark Behrens, Paul Goerss, Hans-Werner Henn, Mike Hopkins, Niko Naumann and Vesna Stojanoska. We also thank the referee and the editors for their input.

\section{Orbits modulo $(p)$}\label{sec:res}
In this section, we prove our main result which is a direct computation of the orbits for the action of $\GG_2$ at height $2$.

\subsection{Background and Results}
We begin by recalling a few facts in order to state our results. We refer the reader to Hazewinkel \cite{MR506881} for more background on formal group laws.

We let $\mathbb{H}_2$ be the Honda formal group law of height $2$. 
The $p$-series of $\mathbb{H}_2$ has the form
\[[p]_{\mathbb{H}_2}(x) = x^{p^2}.\]
The coefficients of $\mathbb{H}_2$ are in $\FF_p$. We let $ \mathcal{O}_2$ be the endomorphism ring of $\mathbb{H}_2$ over $\FF_{p^2}$. Then $\mathcal{O}_2$ is a module over the $p$-adic integers $\Z_p$, generated by the automorphisms
\begin{align*}
[1](x) &= x & S(x) &= x^{p} & \zeta (x) &= \zeta x  
\end{align*}
where $\zeta \in \FF_{p^2}$ is a primitive $p^2-1$th root of unity. In fact, letting $\W = \Z_p(\zeta)$ be the ring of integers of the unramified field extension $\Q_p(\zeta)$ of degree $2$ over $\Q_p$, an explicit presentation of $\mathcal{O}_2$ is given by 
\[\mathcal{O}_2 \cong \W\langle S \rangle / (S^2 = p, Sa = a^{\sigma}S)\]
where $a \in \W$ and $\sigma$ is the Frobenius automorphism in 
\[\Gal = \Gal(\Q_p(\zeta)/\Q_p ) \cong \Z/2.\] 
The group of automorphisms of $\mathbb{H}_2$ is $\SS_2 = \mathcal{O}_2^{\times}$. Since $\Gal$ acts on $\mathcal{O}_2$ via its natural action on $\W$ (and fixing $S$), we can define
\[\GG_2 = \SS_2 \rtimes \Gal.\]

Now, we turn to the description of the Lubin--Tate ring $R_2$. See Lubin--Tate \cite{lubintate} for more details. Let
$R_2 =  \WW[\![u_1]\!]$
and $F(x,y)= x+_F y $ be a deformation of $\mathbb{H}_2$ defined over $R_2$, chosen so that
\[  [p]_F(x) = px +_F u_1 x^p +_F x^{p^2}.\] It follows from Lubin--Tate theory that the deformations of $\mathbb{H}_2$ to complete local rings are co-represented by continuous homomorphisms from the ring $R_2$. The group $\SS_2$ naturally acts on $R_2$. The Galois group acts on $R_2$ via the action on $\W$, fixing $u_1$, and this extends the action of $\SS_2$ to an action of $\GG_2$.

To describe the action of $\SS_2$, note that any element $g \in \SS_2$ can be expressed uniquely as a power series 
\[g=\sum_{i=0}^{\infty} g_i S^i \]
where $g_i^{p^2}-g_i =0$. In other words, a coefficient $g_i$ is either zero or a Teichm\"uller lift of $\FF_{p^2}^{\times}$ in $\W^{\times}$. As we will see in \fullref{sec:lader} below,
\begin{equation}\label{eqn:actgen}
g_*(u_1) = t_0^{p-1}u_1 + t_0^{-1}t_1(p-p^p)
\end{equation}
for a unit $t_0$ in $\WW[\![u_1]\!]$ such that $t_0 = g_0$ modulo $(p,u_1)$ and an element $t_1 \in \WW[\![u_1]\!]$ such that $t_1 = g_1$ modulo $(p,u_1)$. If $g = \zeta$ is a primitive $p^2-1$th root of unity in $\WW^{\times} \subseteq \SS_2$, one can show that $t_0 = \zeta$ and $t_1=0$, so that 
\begin{equation}\label{eqn:actzeta}
\zeta_*(u_1) = \zeta^{p-1}u_1 .
\end{equation}
For more general elements $g \in \SS_2$, $t_0$ is tedious to compute and \fullref{sec:lader} is dedicated to this task.

The goal of this paper is to compute the orbits for the action of $\G_2$ on $R_2/p$, that is, the coinvariants $(R_2/p)_{\G_2}$. We recall the definition of the coinvariants for the action of a profinite group on a profinite module. Let $G = \varprojlim_i G/G_i$ for finite quotients $G/G_i$. Define
\[\Z_p[\![ G]\!] = \varprojlim_{i,j} \Z/p^j [G/G_i] \]
and $\FF_p[\![ G]\!]  = \Z_p[\![ G]\!] /(p)$.
Then, for any profinite module $M =  \varprojlim_k M_k$ where $M_k$ are finite discrete $\Z_p[\![ G]\!] $-modules, we have
\begin{equation*}
M_{G}  =  \varprojlim_{k,j} M_k\otimes_{\Z_p[\![G]\!]} \Z/p^j 
 \end{equation*}
 for the trivial action of $G$ on the right factor $\Z/p^j$. Note that if $M$ is an $\FF_p$-vector space, then
\[M_{G}  \cong  \varprojlim_k M_k \otimes_{\FF_p[\![G]\!]} \FF_p.\] 
When $G= \G_2$ or $\SS_2$, we can choose $G_i$ to be the subgroup consisting of those elements of $\SS_2$ which are congruent to $1$ modulo $(S^i)$. For
 $M = R_2/p$, we can choose $M_k $ to be the discrete finite module $R_2/(p, u_1^k)$ and we have
\begin{equation}\label{eqn:invsys}
(R_2/p)_{G} = \varprojlim_k R_2/(p,u_1^k) \otimes_{\FF_p[\![G]\!]} \FF_p. 
 \end{equation}

We now state the main result.
\begin{thm}\label{thm:mainmodp}
There is an isomorphism 
$(R_2/p)_{\G_2}  \cong \FF_p$ for all primes $p$.
\end{thm}

The proof of \fullref{thm:mainmodp} uses formulas for the action of $\GG_2$. We begin with a summary of the results which are covered in detail in \fullref{sec:lader}.
\subsection{Summary of the action} The action of $\G_2$ on 
\[R_2/p =  \FF_{p^2}[\![ u_1 ]\!] \]
is given by \eqref{eqn:actgen}, modulo a computation of the unit $t_0$. The following result, which is \cite[Corollary 3.4]{lader} for $p\geq 5$ and \cite[Section 4.1]{hkm} for $p=3$, is sufficient for our purposes when $p$ is odd. We will review the proof of this result in \fullref{sec:lader} below and generalize it to include the case $p=2$.
\begin{thm}\label{thm:t0}
Let $p$ be any prime. Let $g \in \SS_2$ be such that $g=1+g_1S+g_2S^2$ modulo $(S^3)$. Then
\[t_0= 1+g_1^pu_1-g_1u_1^p+(g_2-g_2^p)u_1^{p+1}+\sum\limits_{i=1}^{p-1}{\frac{1}{p}}{{p}\choose{i}}g_1^{pi}u_1^{p+1+i} + g_1^2u_1^{2p}+g_1^{p}  u_1^{p^2} \mod (p, u_1^{2p+1}).\]
\end{thm}
When $p=2$, we will need more information about the action of $g$. We give a computer assisted
 proof of the following result in \fullref{sec:lader}.\footnote{If one is willing to work with the formal group law of a super-singular elliptic curve rather than the Honda formal group law, an analogue of \fullref{thm:actp2} follows from Section 6 of \cite{Paper2} where the results were obtained directly. The analogue of \fullref{thm:mainmodp} also holds in this case, the proof being completely analogous to the one provided below.}
 \begin{thm}\label{thm:actp2}
Let $p=2$. If $g = 1+g_2S^2 +g_3S^3+g_4S^4+\ldots $, then  
\[t_0 = 1+ (g_2+g_2^2)u_1^3+g_3u_1^5+g_3u_1^8+(g_4+g_4^2)u_1^9   \mod (2, u_1^{10}) .\]
\end{thm}

\subsection{Prime independent arguments} 
The bulk of the proof of \fullref{thm:coinvSS} will be in proving the following proposition. We abbreviate $\SS=\SS_2$ and $R=R_2$ and let $[x]$ denote the image of an element $x$ under the natural map
$\FF_{p^2}[ u_1 ]/(u_1^k) \to (\FF_{p^2}[ u_1 ]/(u_1^k))_{\SS}$.

\begin{prop}\label{prop:keyres}
For $k \geq 2$, $[u_1^{k-1}] =0$ in $ (\FF_{p^2}[ u_1 ]/(u_1^{k}) )_{\SS} $.
\end{prop}

Assuming \fullref{prop:keyres}, we prove the following result, which immediately implies \fullref{thm:mainmodp} by taking Galois coinvariants since $(\FF_{p^2})_{\Gal} \cong \FF_p$.
\begin{prop} \label{thm:coinvSS}
The quotient map $\FF_{p^2}[\![ u_1]\! ] \to \FF_{p^2}$ induces an isomorphism
\[(\FF_{p^2}[\![ u_1]\! ] )_{\SS} \cong \FF_{p^2}.\]
\end{prop}
\begin{proof}
Since taking coinvariants is a right exact functor, the maps in the inverse system \eqref{eqn:invsys} fit into an exact sequence
\[ \xymatrix{  ((u_1^{k-1})/(u_1^{k}))_{\SS}   \ar[r] & (\FF_{p^2}[u_1 ]/(u_1^{k}))_{\SS}  \ar[r] &  (\FF_{p^2}[u_1 ]/(u_1^{k-1}) )_{\SS}  \ar[r]&  0 }  \] 
and \fullref{prop:keyres} implies that the left map is trivial. Therefore, \eqref{eqn:invsys} is a constant inverse system whose first term is $\FF_{p^2}$.\end{proof}

We turn to the proof of \fullref{prop:keyres}. We begin with a simple result. 
\begin{prop}\label{prop:roots}
If $n$ is not of the form $(p+1)\tr$, then for all $k \geq 0$,  $[u_1^n]=0$ in $(\FF_{p^2}[ u_1 ] /(u_1^{k}) )_{\SS}$.
\end{prop}
\begin{proof}
By \eqref{eqn:actzeta}, 
\begin{align*}
\zeta_*(u_1^n) &= \zeta^{n(p-1)} u_1^n = u_1^n + (\zeta^{n(p-1)}-1) u_1^n.
\end{align*}
Therefore, $ (\zeta^{n(p-1)}-1) [u_1^n]=0$. Since $\zeta$ is a primitive $p^2-1$th root of unity, then $\zeta^{n(p-1)}-1$ is a unit in $\FF_{p^2}$ provided that $p+1$ does not divide $n$. It follows that, in this case, $[u_1^n]=0$.
\end{proof}

\begin{rem}
Note that this result is stronger than \fullref{prop:keyres} in the case $n=k-1$. It will be used in its full strength in our proof of \fullref{prop:keyres}.
\end{rem}

The technique for showing that $[u_1^{k}]=0$ in $(\FF_{p^2}[u_1]/(u_1^{k+1}))_{\SS}$ for $k=(p+1)\tr$ varies on the $p$-adic expansion of $\tr$. 
\begin{prop}\label{prop:pplusone}
If $k=(p+1)\tr$ for $\tr$ non trivial such that $\tr \neq 1$ modulo $(p)$, then $[u_1^{k}]=0$ in $(\FF_{p^2}[u_1]/(u_1^{k+1}))_{\SS}$.
\end{prop}
\begin{proof}
Let $g= 1+S$.
It follows from \fullref{thm:t0} that 
\[t_0=1+u_1 \mod (p, u_1^{p}).\]
Therefore by \eqref{eqn:actgen}
\begin{align*}
g_*(u_1^{k-1}) &= u_1^{k-1}(1+u_1)^{(p-1)(k-1)} \mod (p, u_1^{p+k-1}) \\
&= u_1^{k-1}+ (p-1)(k-1)u_1^{k} \mod (p, u_1^{k+1}).
\end{align*}
So long as $k\neq 1$ modulo $(p)$, we can conclude that $[u_1^{k}]=0$ in $(\FF_{p^2}[u_1]/(u_1^{k+1}))_{\SS}$. Since $k= \tr$ modulo $(p)$, this proves the claim.
\end{proof}

\subsection{The remainder of the argument for odd primes} 
Now, we fix $p$ odd. The case $p=2$ will be treated below. We let
\begin{equation}\label{eqn:padic}
k = (p+1)(1 + p + p^2 + \ldots + p^{\ell-1} + p^\ell \eta) 
\end{equation}
for $\ell \geq 0$ and $\eta$ a non-negative integer such that $\eta  \neq 1$ modulo $(p)$. The complexity of the problem depends on $\ell$. The case when $\ell=0$ was \fullref{prop:pplusone}, so we now turn to the case when $\ell \geq 1$ in \eqref{eqn:padic}. Let
\begin{equation}\label{eqn:kr}
k_{r} = \begin{cases}  k & r=0 \\
k_{r-1}-(p+1)p^{r-1} & 1\leq r < \ell-1
\end{cases} \end{equation}
for $0 \leq r < \ell-1$. 

We prove that $[u_{1}^{k_r}] =- [u_1^{k_{r+1}}]$ for $0\leq r < \ell-1$ (\fullref{prop:stepsodd}) and $[u_1^{k_{\ell-1}}]=0$ (\fullref{prop:lastodd}) in $(\FF_{p^2}[u_1]/(u_1^{k+1}))_{\SS}$. Together, these results finish the proof of \fullref{prop:keyres}.

\begin{prop}\label{prop:stepsodd}
Let $k_r$ be as in \eqref{eqn:kr}. For $0\leq r < \ell-1$, 
\[[u_{1}^{k_r}] =- [u_1^{k_{r+1}}]\]
in $(\FF_{p^2}[u_1]/(u_1^{k+1}))_{\SS}$.
\end{prop}
\begin{proof}
From \fullref{thm:t0}, we deduce that for $g=1+S$, 
\[t_0=1+u_1-u_1^p+\sum\limits_{i=1}^{p-1}{\frac{1}{p}}{{p}\choose{i}}u_1^{p+1+i} + u_1^{2p}\quad \mod (u_1^{2p+1}).\]
We use \eqref{eqn:actgen}, the fact that $a^p=a$ for $a\in \FF_p$, $(x+y)^p = x^p+y^p$ modulo $(p)$ and the fact that 
\[k_{r+1}-p^r = p^{r} \tr_r\]
where $\tr_r = (p+1)(p+\ldots+p^{\ell-1-r}+p^{\ell-r}\eta)-1$. With this, we deduce that
\begin{align*}
g_*(u_1^{k_{r+1}-p^r}) &= u_1^{k_{r+1}-p^r}\left(1+u_1^{p^{r}}-u_1^{p^{r+1}}+\sum\limits_{i=1}^{p-1}{\frac{1}{p}}{{p}\choose{i}} u_1^{p^{r+1}+(1+i)p^{r}} + u_1^{2p^{r+1}}\right)^{(p-1)\tr_r} 
\end{align*}
modulo $(u_1^{k_{r+1} +2p^{r+1}})$. We now simplify this equation.
We compute modulo $(u_1^{k+1})$ and note that
\begin{align*}
k+1 &= k_{r+1} + (p+1)(1+\ldots+p^r) +1 \\
& = k_{r+1} + 2(1+\ldots+p^r)+p^{r+1} <  k_{r+1}+3p^r+p^{r+1}.
\end{align*}
Therefore, we immediately get rid of all terms of the form $u_1^n$ for $n \geq k_{r+1}+3p^r+p^{r+1}$.
Next, we use the fact that $\tr_r= p-1$ modulo $(p^2)$ so that $(p-1)\tr_r = 1+p(p-2)$ modulo $(p^2)$. For $i=i_0+pi_1<p^2$ with $0\leq i_0,i_1 \leq p-1$, we then have 
\[\binom{(p-1)\tr_r}{i} 
= \binom{(p-1)^2}{i} 
=   \binom{1}{i_0} \binom{p-2}{i_1} \mod (p) ,\] 
where $\binom{m}{n}=0$ if $m<n$. In particular, $\binom{(p-1)^2}{2} = 0 $ modulo $(p)$.
Combining these facts, we obtain:
\begin{align*}
g_*(u_1^{k_{r+1}-p^r}) &= u_1^{k_{r+1}-p^r}(1+u_1^{p^{r}}-u_1^{p^{r+1}}+ u_1^{p^{r+1}+2p^{r}}+\frac{p-1}{2} u_1^{p^{r+1}+3p^{r}} + u_1^{2p^{r+1}})^{(p-1)\tr_r} \\
&= u_1^{k_{r+1}-p^r}\left(1  + \sum_{i=1}^{p+3} \binom{(p-1)\tr_r}{i} \left(u_1^{p^{r}}-u_1^{p^{r+1}}+ u_1^{p^{r+1}+2p^{r}}+\frac{p-1}{2} u_1^{p^{r+1}+3p^{r}} + u_1^{2p^{r+1}}\right)^i\right) \\
&= u_1^{k_{r+1}-p^r}\left(1  + \sum_{i=1}^{p+3} \binom{(p-1)^2}{i} \left(u_1^{p^{r}}-u_1^{p^{r+1}}+ u_1^{p^{r+1}+2p^{r}}+\frac{p-1}{2} u_1^{p^{r+1}+3p^{r}} + u_1^{2p^{r+1}}\right)^i\right) \\
&=u_1^{k_{r+1}-p^r}\\
&\ \ \ +( u_1^{k_{r+1}}-u_1^{k_{r+1}+p^{r+1}-p^r}+ u_1^{k_{r+1}+p^{r+1}+p^{r}}+\frac{p-1}{2} u_1^{k_{r+1}+p^{r+1}+2p^{r}} + u_1^{k_{r+1}+2p^{r+1}-p^r}) \\
&\ \ \ +\binom{(p-1)^2}{3}( u_1^{k_{r+1}+2p^{r}}-3u_1^{k_{r+1}+p^{r+1}+p^r})\\
&\ \ \ +\binom{(p-1)^2}{4}( u_1^{k_{r+1}+3p^{r}}-4u_1^{k_{r+1}+p^{r+1}+2p^r})\\
&\ \ \ + \sum_{i=5}^{p+3} \binom{(p-1)^2}{i}u_1^{k_{r+1}+p^{r}(i-1)}.
\end{align*}
Note further that, if $p\neq 3$, then $\binom{(p-1)^2}{3}=0$ modulo $(p)$. So $3\binom{(p-1)^2}{3}=0$ modulo $(p)$ for all primes.
The above computation then gives the following relation in the coinvariants
\begin{align*}
0&= [u_1^{k_{r+1}}] +[u_1^{k_{r+1}+p^{r+1}+p^r}] \\
& \ \ \   -[u_1^{k_{r+1}+p^{r+1}-p^r}] + [u_1^{k_{r+1}+2p^{r+1}-p^r}] \\
&  \ \ \  +\left(\frac{p-1}{2} -4\binom{(p-1)^2}{4}\right) [u_1^{k_{r+1}+p^{r+1}+2p^r}]  + \sum_{i=3}^{p+3}  \binom{(p-1)^2}{i}[u_1^{k_{r+1}+p^{r}(i-1)}]  .
\end{align*}
Recall that $[u_1^{n}]=0$ if $n$ is not a multiple of $p+1$. Since $k_{r+1}$ is a multiple of $p+1$, it follows that $[u_1^{k_{r+1}+p^{r+1}-p^r}] = [u_1^{k_{r+1}+p^{r+1}+2p^r}] = [u_1^{k_{r+1}+2p^{r+1}-p^r}]= 0$ modulo $(p)$. Similarly, the only term that can remain in the summation after taking this into account is the case $i=p+2$. However, $\binom{(p-1)^2}{p+2}=0$ modulo $(p)$, so the summation is also zero. Therefore, the second and third lines of the equation are zero in the coinvariants. 
We conclude that
\begin{align*}
0&= [u_1^{k_{r+1}}]+ [u_1^{k_{r+1}+p^{r+1}+p^{r}}].\qedhere
\end{align*}
\end{proof}

\begin{prop}\label{prop:lastodd}
Let $k_{\ell-1} = (p+1)(p^{\ell-1}+p^{\ell}\eta)$ for a non-negative integer $\eta$ such that $\eta \neq 1$ modulo $(p)$ as in \eqref{eqn:kr}. Then
\[[u_{1}^{k_{\ell-1}}] =0 \]
in $(\FF_{p^2}[u_1]/(u_1^{k+1}))_{\SS}$.
\end{prop}
\begin{proof}
Note that
\[k_{\ell-1}-p^{\ell} = p^{\ell-1}\tr_{\ell}\]
where $\tr_{\ell}=1+(p+p^2)\eta $. Therefore,
\begin{align*}
g_*(u_1^{k_{\ell-1}-p^{\ell}})&=u_1^{k_{\ell-1}-p^\ell}
\left(1+u_1^{p^{\ell-1}}-u_1^{p^{\ell}}+\sum\limits_{i=1}^{p-1}{\frac{1}{p}}{{p}\choose{i}}u_1^{p^{\ell}+(1+i)p^{\ell-1}} + u_1^{2p^{\ell}}\right)^{(p-1)\tr_{\ell}}
\end{align*}
modulo $ (u_1^{ k_{\ell-1}+p^{\ell}+ p^{\ell-1} } )$ and note that 
\[k+1 =k_{\ell-1}+ p^{\ell-1}+2(1+\ldots+p^{\ell-2}) <k_{\ell-1}+p^{\ell-1}+3p^{\ell-2} \leq k_{\ell-1}+2p^{\ell-1}.\] So, using the fact that $(p-1)\tr_\ell = (p-1)+p(p-\eta)$ modulo $(p^2)$,
we simplify as before to obtain
\begin{align*}
g_*(u_1^{k_{\ell-1}-p^{\ell}})&=u_1^{k_{\ell-1}-p^\ell}
(1+u_1^{p^{\ell-1}}-u_1^{p^{\ell}}+u_1^{p^{\ell}+2p^{\ell-1}} )^{(p-1)\tr_{\ell}}\\
&=u_1^{k_{\ell-1}-p^\ell}
(1+(p-1)\tr_\ell(u_1^{p^{\ell-1}}-u_1^{p^{\ell}}+u_1^{p^{\ell}+2p^{\ell-1}} ) + \binom{(p-1)\tr_{\ell}}{2}(u_1^{p^{\ell-1}}-u_1^{p^{\ell}} )^2)\\
&= u_1^{k_{\ell-1}-p^\ell}\big(1 - u_1^{p^{\ell-1}}+u_1^{p^{\ell}}-u_1^{p^{\ell}+2p^{\ell-1}} + \binom{p-1}{2} (u_1^{2p^{\ell-1}} -2 u_1^{p^{\ell} +p^{\ell-1}}) \\
& \ \ \ +\sum_{i=3}^{p+1} \binom{(p-1)+p(p-\eta)}{i} u_1^{i p^{\ell-1}} \big) \\
&= u_1^{k_{\ell-1}-p^\ell} - u_1^{k_{\ell-1}-p^\ell+p^{\ell-1}}+u_1^{k_{\ell-1}}-u_1^{k_{\ell-1}+2p^{\ell-1}} + \binom{p-1}{2} (u_1^{k_{\ell-1}-p^\ell+2p^{\ell-1}} -2 u_1^{k_{\ell-1} +p^{\ell-1}}) \\
& \ \ \ +\sum_{i=3}^{p-1} \binom{p-1}{i} u_1^{k_{\ell-1}-p^\ell+i p^{\ell-1}}   +\sum_{j=0}^{1}  \binom{p-1}{j}\binom{p-\eta}{1} u_1^{k_{\ell-1}+jp^{\ell-1}} 
\end{align*}
As before, noting that $[u_1^n]=0$ if $p+1$ does not divide $n$, while $p+1$ does divide $k_{\ell-1}$, we obtain the following relation in the coinvariants:
\begin{align*}
0&= - [u_1^{k_{\ell-1}-p^\ell+p^{\ell-1}}]+[u_1^{k_{\ell-1}}]-[u_1^{k_{\ell-1}+2p^{\ell-1}}] + \binom{p-1}{2}( [u_1^{k_{\ell-1}-p^\ell+2p^{\ell-1}}] -2 [u_1^{k_{\ell-1} +p^{\ell-1}}]) \\
& \ \ \ +\sum_{i=3}^{p-1} \binom{p-1}{i} [u_1^{k_{\ell-1}-p^\ell+i p^{\ell-1}}]   -\eta \sum_{j=0}^{1}  \binom{p-1}{j}[u_1^{k_{\ell-1}+jp^{\ell-1}}]  \\
&= [u_1^{k_{\ell-1}}] -\eta  [u_1^{k_{\ell-1}}] 
\end{align*}
Since $\eta \neq 1$ modulo $(p)$, we can conclude that  $[u_1^{k_{\ell-1}}] =0$.
\end{proof}

\subsection{The remainder of the argument for the prime two} 
We follow steps similar to those taken at odd primes.
\begin{prop}
Let $k= 3\tr$ where $\tr$ is an integer congruent to $1$ modulo $(4)$. Then $[u_1^{k}]=0$ in $(\FF_{4}[u_1]/(u_1^{k+1}))_{\SS}$.
\end{prop}
\begin{proof}
Choose $g=1+S$ and consider $g_*(u_1^{k-2})$. We have
\begin{align*}
g_*(u_1^{k-2}) &=  t_0^{k-2}u_1^{k-2} \\
&=  u_1^{k-2}(1+ u_1 +u_1^2 )^{k-2} \mod (2, u_1^{k+1}) \\
&=  u_1^{k-2} + u_1^{k-1} +\left(1+ \binom{k-2}{2} \right)u_1^{k} \mod (2, u_1^{k+1}).
\end{align*}
Since $k-2 = 1$ modulo $(4)$, we have
\begin{align*}
\binom{k-2}{2}  = \binom{1}{0} \binom{0}{1} \ldots  = 0 \mod (2).
\end{align*}
So, $[u_1^{k-1}]= [u_1^{k}]$ in $(\FF_{4}[u_1]/(u_1^{k+1}))_{\SS}$. 
Since $k-1\neq 0$ modulo $(3)$, $[u_1^{k-1}]$ is zero in $(\FF_{4}[u_1]/(u_1^{k+1}))_{\SS}$ by \fullref{prop:roots} and the claim follows.
\end{proof}

Now, we let $k= 3(1 + 2 + 2^2+ \ldots + 2^{\ell-1} + 2^{\ell+1} \eta)$ for $\eta$ any non-negative integer and $\ell\geq 2$, we write
\begin{equation}\label{eqn:krtwo}
k_{r} = \begin{cases} k  & r=0 \\
k_{r-1}-3 \cdot 2^{r-1} & 1\leq r \leq \ell-2.
\end{cases} \end{equation}
We prove that $[u_{1}^{k_r}] =  [u_1^{k_{r+1}}]$ for $0\leq r < \ell-2$ (\fullref{prop:stepstwo}) and that $[u_{1}^{k_{\ell-2}}] =  0$ (\fullref{prop:finaltwo}) in $(\FF_{4}[u_1]/(u_1^{k+1}))_{\SS}$. Together, these results finish the proof of \fullref{prop:keyres}.

\begin{prop}\label{prop:stepstwo}
Let $k_r$ be as in \eqref{eqn:krtwo}. For $0\leq r < \ell-2$, 
\[[u_{1}^{k_r}] =  [u_1^{k_{r+1}}]\]
in $(\FF_{4}[u_1]/(u_1^{k+1}))_{\SS}$.
\end{prop}
\begin{proof}
Take $g = 1+\zeta S^2 + \zeta S^4$. Note that $\zeta+\zeta^2 =1$ modulo $(2)$, so by \fullref{thm:actp2},
\begin{align*}
g_*(u_1^{k_{r+1}-3\cdot 2^{r}}) &= u_1^{k_{r+1}-3\cdot 2^{r}} (1+u_1^3+u_1^9)^{k_{r+1}-3\cdot 2^{r}}  \mod (2, u_1^{2^{r+3}+k_{r+1}- 2^{r}})
\end{align*}
Note that since
\[k+1-k_{r+1} = 3(1+2+\ldots+2^r)+1 =3 \cdot 2^{r+1}-2\]
and $2^{r+3}-2^r >3 \cdot 2^{r+1}-2$, we have that $2^{r+3}+k_{r+1}- 2^{r} \geq k+1$. So modulo $(u_1^{k+1})$,
\begin{align*}
g_*(u_1^{k_{r+1}-3\cdot 2^{r}}) 
 &= u_1^{k_{r+1}-3\cdot 2^{r}} (1+u_1^{3 \cdot 2^r}+u_1^{9\cdot 2^r})^{2^{-r}k_{r+1}-3} \\
 &= \sum_{i=0}^{2^{-r}k_{r+1}-3}\binom{2^{-r}k_{r+1}-3}{i} u_1^{k_{r+1}-3\cdot 2^{r}}  (u_1^{3 \cdot 2^r}+u_1^{9\cdot 2^r})^{i} \\
    &= \sum_{i=0}^{2^{-r}k_{r+1}-3}\sum_{j=0}^{i} \binom{2^{-r}k_{r+1}-3}{i} \binom{i}{j} u_1^{k_{r+1} + 3 \cdot 2^{r+1}j+3 \cdot 2^r(i-1)}  \ . \\
\intertext{Modulo $(u_1^{k+1})$, only terms with $j=0$ and $i< 3$ contribute to the sum. So}
  &=  \binom{2^{-r}k_{r+1}-3}{0}  u_1^{k_{r+1} -3 \cdot 2^r}+ \binom{2^{-r}k_{r+1}-3}{1}  u_1^{k_{r+1} } +\binom{2^{-r}k_{r+1}-3}{2}   u_1^{k_{r+1} +3 \cdot 2^r} \\
    &=   u_1^{k_{r+1} -3 \cdot 2^r}+ u_1^{k_{r+1} } +\binom{2^{-r}k_{r+1}-3}{2}   u_1^{k_{r}} \ . 
    \end{align*}
Finally, since $2^{-r} k_{r+1}-3  = 3$ modulo $(4)$, then $\binom{2^{-r}k_{r+1}-3}{2} = 1$ modulo $(2)$. So we conclude that 
\[ g_*(u_1^{k_{r+1}-3\cdot 2^{r}})  = u_1^{k_{r+1} -3 \cdot 2^r}+ u_1^{k_{r+1} } +  u_1^{k_{r}} \mod (2, u_1^{k+1}).\]
Therefore, $[u_1^{k_{r+1} }] = [ u_1^{k_{r}} ]$ as desired.
\end{proof}

\begin{prop}\label{prop:finaltwo}
Let $k_{\ell-2}$ be as in \eqref{eqn:krtwo}. Then $[u_{1}^{k_{\ell-2}}] =  0$ in $(\FF_{4}[u_1]/(u_1^{k+1}))_{\SS}$.
\end{prop}
\begin{proof}
Choose $g=1+S$ and consider $g_*(u_1^{2^{\ell-1} + 3 \cdot 2^{\ell+1}\eta})$. We have
\begin{align*}
g_*(u_1^{2^{\ell-1} + 3 \cdot 2^{\ell+1}\eta}) &= u_1^{2^{\ell-1} + 3 \cdot 2^{\ell+1}\eta}(1+ u_1 +u_1^2+ u_1^4 )^{2^{\ell-1} + 3 \cdot 2^{\ell+1}\eta} \mod (2,u_1^{3(2^{\ell}    +   2^{\ell+1}\eta )} )
\end{align*}
noting that $3(2^{\ell}    +   2^{\ell+1}\eta ) \geq k+1$.
Therefore, modulo $(u_1^{k+1})$, we have
\begin{align*}
g_*(u_1^{2^{\ell-1} + 3 \cdot 2^{\ell+1}\eta})  &= u_1^{2^{\ell-1} + 3 \cdot 2^{\ell+1}\eta}(1+ u_1^{2^{\ell-1}} +u_1^{2^{\ell}}+ u_1^{2^{\ell+1}} )^{1+ 3 \cdot 2^{2}\eta}  \\
&= u_1^{2^{\ell-1} + 3 \cdot 2^{\ell+1}\eta}\sum_{s=0}^{1+ 3 \cdot 2^{2}\eta} \sum_{i=0}^{s} \sum_{j=0}^{i} \binom{1+ 3 \cdot 2^{2}\eta}{s}\binom{s}{i}   \binom{i}{j} u_1^{2^{\ell-1}(s-i)} u_1^{2^{\ell}(i-j)} u_1^{2^{\ell+1}j} \\
&= \sum_{s=0}^{1+ 3 \cdot 2^{2}\eta} \sum_{i=0}^{s} \sum_{j=0}^{i} \binom{1+ 3 \cdot 2^{2}\eta}{s}\binom{s}{i}   \binom{i}{j} u_1^{2^{\ell-1}(1 +s+i+2j)+ 3 \cdot 2^{\ell+1}\eta} .
\end{align*}
Note that if $ s+i+2j \geq 5$, the terms vanish for degree reasons. Hence, $ s \leq 4$, $ i \leq 4-s$, and $ j \leq 2-(s+i)/2$, so that
\begin{align*}
g_*(u_1^{2^{\ell-1} + 3 \cdot 2^{\ell+1}\eta})  &= \sum_{s=0}^{4} \sum_{i=0}^{\max(s, 4-s)} \sum_{j=0}^{\max(i, 2-(s+i)/2)} \binom{1+ 3 \cdot 2^{2}\eta}{s}\binom{s}{i}   \binom{i}{j} u_1^{2^{\ell-1}(1 +s+i+2j)+ 3 \cdot 2^{\ell+1}\eta} .
\end{align*}
Since $1+3 \cdot 2^2 \eta= 1$ modulo $(4)$, $ \binom{1+ 3 \cdot 2^{2}\eta}{2} = \binom{1+ 3 \cdot 2^{2}\eta}{3} = 0$ modulo $(2)$. Further, $\binom{1+ 3 \cdot 2^{2}\eta}{4}  = \eta$ modulo $(2)$. Enumerating the remaining possibilities gives
\begin{align*}
g_*(u_1^{2^{\ell-1} + 3 \cdot 2^{\ell+1}\eta})  &= u_1^{2^{\ell-1}+ 3 \cdot 2^{\ell+1}\eta} + u_1^{2^{\ell}+ 3 \cdot 2^{\ell+1}\eta} + u_1^{3( 2^{\ell-1}+  2^{\ell+1}\eta)} +(1+\eta) u_1^{3 ( 2^{\ell}+ 2^{\ell+1}\eta) -2^{\ell-1}} 
\end{align*}
and this holds modulo $ (2, u_1^{k+1})$. Therefore, since $k_{\ell-2} =3 (2^{\ell-1}+   2^{\ell+1}\eta)$, we have
\[ [u_1^{k_{\ell-2}} ] = [u_1^{2^{\ell}+3 \cdot 2^{\ell+1}\eta} ] +(1+\eta)[u_1^{3 ( 2^{\ell} +  2^{\ell+1}\eta)-2^{\ell-1}}]  \] in $(\FF_{4}[ u_1]/(u_1^{k+1}))_{\SS}$. However, the right hand terms are zero by \fullref{prop:roots}, which proves the claim.
\end{proof}

\section{The action of the Morava stabilizer group}\label{sec:lader}

In this section, we continue to abbreviate $R=R_2$ and $\SS= \SS_2$. Here, we follow the derivation of the formula for the universal deformation $F(x,y)$ and the resulting formulas for the action of $\SS$ on $R$ as outlined in the doctoral thesis of Lader \cite{lader}. Note that the methods of \cite{lader} are a generalization of the techniques of \cite[Section 4.1]{hkm} at primes $p\geq 5$. We claim no originality, but include the computations we need here. One reason for this is that the doctoral thesis is only available in French. We have also decided to add more details to the proofs and have taken the opportunity to note that, with one minor adjustment (\fullref{thm:include2}), the results generalize to the case $p=2$.

\subsection{The universal deformation} \label{sec:unidefcomp}
We start by fixing a prime $p$. Let
\[V \cong \Z_{(p)}[v_1,v_2, \ldots]\]
and $G(x,y)$ be the universal $p$-typical formal group law defined over $V$. We choose to work with the Araki generators, which can be described as follows.
Let $\ell_0=1$,
\[	\log_G (x) =   \sum_{i\geq 0} \ell_i x^{p^i} \]
and $\exp_G(x)$ be the formal power series inverse of $\log_G(x)$ under composition. The Araki generators $v_i$'s are then determined by the recursive formula
\[p \ell_n = \sum_{0 \leq i \leq n} \ell_i v_{n-i}^{p^i} \]
where by convention $v_0 =p$. 
The universal $p$-typical formal group law is computed as
\[G(x,y) = \exp_G(\log_G(x)+\log_G(y))\]
and the Araki generators have the property that 
\[[p]_G(x) = \sum_{i\geq 0}{}^{G} v_i x^{p^i}\]
where $[p]_G(x)$ is the $p$-series for $G(x,y)$.

The formal group law $F(x,y)$ over $R$ is obtained from $G(x,y)$ as follows. Consider the 
ring homomorphism
\[\varphi \colon V \to R\]
determined by $\varphi(v_1) = u_1 $, $\varphi(v_2) = 1$ and $\varphi(v_n) =0$ for $n>2$. The formal group law $F(x,y)$ is defined by
\[F(x,y) = \varphi_*G(x,y).\]
It follows that 
\[	\log_F (x) =   \sum_{i\geq 0} L_i x^{p^i} \]
for $L_i = \varphi(\ell_i)$
and that
\[[p]_F(x) = px +_{F} u_1 x^p +_{F} x^{p^2}.\]
We record that
\begin{equation}\label{eq:Ls}	L_1 = \frac{u_1}{p-p^p}, \qquad L_2 = \frac{ \left( 1 + \frac{u_1^{p+1}}{p-p^p} \right)}{p-p^{p^2}}, \qquad L_3=\frac{\frac{u_1}{p-p^p}+\frac{u_1^{p^2}}{p-p^{p^2}}+\frac{u_1^{p^2+p+1}}{\left(p-p^p\right) \left(p-p^{p^2}\right)}}{p-p^{p^3}}. \end{equation}

The goal of this section is to approximate $F(x,y)$. From now on, we let $\log(x) =\log_F(x)$ and $\exp(x) =\exp_F(x)$ so that
\[F(x,y) = \exp(\log(x)+\log(y)).\]
We will prove the following result, which is \cite[Lemma 3.1]{lader}.
\begin{theorem}[Lader]\label{thm:Fform} Let $p$ be any prime. Modulo $(x,y)^{p^2+1}$, the formal group law $F(x,y)$ satisfies
	\[	F(x,y) = x + y - \frac{u_1}{1-p^{p-1}} C_p (x,y) - \sum_{i=1}^p u_1^{i+1} P_{p+i(p-1)} (x,y) - \frac{1}{1-p^{p^2-1}}\left( 1 + \frac{u_1^{p+1}}{p - p^p} \right) C_{p^2} (x,y),\]
where
	\[	C_{p^k} (x,y) = \frac{1}{p} \big( (x+y)^{p^k} - x^{p^k} - y^{p^k} \big) \]
and 
	\[	P_{p+i(p-1)} (x,y) = \frac{1}{(p - p^p)^{i+1}} \sum_{j=0}^i \frac{(-1)^j}{j+1} \binom{p(j+1)}{j} \binom{j(p-1)+p}{i-j} (x^p + y^p)^{i-j} (x+y)^{p+pj-i}. \]
\end{theorem}

We begin with some preliminary results.

\begin{theorem}\label{thm:expseries} Let $p$ be any prime. Given 
		\[	\log(x) = x + L_1 x^p + L_2 x^{p^2} \mod (x^{p^3}) \]
	the inverse series is given by
				\[ \exp(x) =  x - L_1 x^p \left( \sum_{j=0}^p \frac{(-1)^{j}}{j+1} \binom{p(j+1)}{j}\left( L_1 x^{p-1} \right)^j \right) - L_2 x^{p^2} \mod (x^{p^2+1}) \]
\end{theorem}
\begin{proof}  First, we recall the Lagrange inversion formula for the inverse of a formal 	power series.  The formula states that given a formal power series
		\[	f(x) = a_1 x + a_2 x^2 + a_3 x^3 + \cdots \]
	the inverse series is given by
		\[	f^{-1}(x) = b_1 x+ b_2 x^2 + b_3 x^3  + \cdots \]
	where $b_1 = \frac{1}{a_1}$ and for $n > 1$ we have
		\[	b_n = \frac{1}{n a_1^n} \sum_{c_1, c_2, c_3, \ldots} (-1)^{c_1+c_2+c_3 + \cdots} \frac{n(n+1)\cdots (n-1+c_1 + c_2  +c_3+ \cdots)}{c_1! c_2 ! c_3 ! \cdots} \left( \frac{a_2}{a_1} \right)^{c_1} \left( \frac{a_3}{a_1} \right)^{c_2} \left( \frac{a_4}{a_1} \right)^{c_3} \cdots, \]
	with the sum taken over $c_1, c_2, c_3, \ldots$ such that
		\[	c_1 + 2c_2 + 3c_3 + \cdots = n-1. \]
	In the current case, $f(x) = \log(x) = x + L_1 x^p + L_2 x^{p^2}$ modulo $(x^{p^3})$, so we have coefficients $a_n$ given by
		\[	a_n = \begin{cases}1 & \text{if } n = 1 \\
			L_1 & \text{if } n = p \\
			L_2 & \text{if } n = p^2 \\
			0 & \text{otherwise}. \end{cases} \]
	The Lagrange inversion formula for the coefficients of $\exp(x)$ then simplifies to
		\[	b_n = \frac{1}{n} \sum_{c_1, c_2, c_3, \ldots} (-1)^{c_1 + c_2 + c_3 + \cdots} \frac{n(n+1) \cdots (n-1+c_1 + c_2 + c_3 + \cdots )}{c_1! c_2! c_3 ! \cdots} a_2^{c_1} a_3^{c_2} a_4^{c_3} \cdots. \]
	But since $a_i = 0$ unless $i = 1, p$ or $p^2$, the terms in the sum will vanish if the exponent on any of these terms is nonzero.  Hence, the only nonzero $c_i$ that will contribute to the sum are those for which $i = p - 1$ and $i = p^2 - 1$.  Hence, 
		\[	b_n = \frac{1}{n} \sum_{c_{p-1}, c_{p^2-1}} (-1)^{c_{p-1} + c_{p^2-1}} \frac{n(n+1) \cdots (n-1 + c_{p-1} + c_{p^2-1})}{c_{p-1}! c_{p^2-1}!} a_p^{c_{p-1}} a_{p^2}^{c_{p^2-1}}\]
	where
		\[	(p-1)c_{p-1} + (p^2-1)c_{p^2-1} = n-1. \]
	We consider $\exp(x)$ modulo $(x^{p^2+1})$, so we only need to compute the coefficients $b_n$ up to $n = p^2$.  Therefore, the only solutions of $(p-1)c_{p-1} + (p^2-1)c_{p^2-1} = n-1$ are 
		\[	c_{p-1} = i, \ c_{p^2-1} = 0 \qquad \text{for \ \ } i = 1, 2, \ldots, p+1\]
	when $n =  i  (p-1) + 1$ together with
		\[	c_{p-1} = 0, \ c_{p^2-1} = 1 \]
	when $n = p^2$. Hence,
		\[	b_n = \begin{dcases}
				\frac{(-1)^{i}}{  i  (p-1) + 1} \binom{pi}{i} L_1^{i} & \text{if } n = i  (p-1)+1 \ \text{ for } i = 1, \ldots, p \\
				\frac{(-1)^{p+1}}{p^2} \binom{p(p+1)}{(p+1)} L_1^{p+1} - L_2 & \text{if } n = p^2 \\
				0 & \text{otherwise}.
			\end{dcases}
			\]

It follows that
			\begin{align*}	
			\exp(x) &= x + \left(\sum_{i=1}^{p+1}\frac{(-1)^{i}}{  i  (p-1) + 1} \binom{pi}{i} L_1^{i}x^{ i  (p-1)+1} \right) - L_2 x^{p^2} \\
			&= x + \left(\sum_{i=1}^{p+1}\frac{(-1)^{i}}{  i  } \binom{pi}{i-1} L_1^{i}x^{ i  (p-1)+1} \right) - L_2 x^{p^2}  
		\end{align*}
		Letting $j = i-1$ gives the formula.
\end{proof}

Now that we have formulas for both $\log(x)$ and $\exp(x)$ we can apply them to compute $F(x,y) = \exp(\log(x) + \log(y))$ and prove \fullref{thm:Fform}.
\begin{proof}[Proof of \fullref{thm:Fform}]  
	First, we consider the middle term $ \sum_{j=0}^p \frac{(-1)^j}{j+1} \binom{p(j+1)}{j}L_1^{j+1} x^{j(p-1)+p}$ of \fullref{thm:expseries}.  Evaluating at $\log (x) + \log (y)$, modulo $(x,y)^{p^2+1}$ we have
	\begin{align*}
		& \sum_{j=0}^p \frac{(-1)^j}{j+1} \binom{p(j+1)}{j} L_1^{j+1} \big( (x+y) + L_1 (x^p + y^p)\big)^{j(p-1)+p} \\
		&= \sum_{j=0}^p \sum_{k=0}^{j(p-1)+p} \frac{(-1)^j}{j+1} \binom{p(j+1)}{j} \binom{j(p-1)+p}{k} L_1^{k+j+1}  (x^p + y^p)^k (x+y)^{j(p-1)+p-k} \\
		\intertext{The terms of this polynomial are homogeneous of degree $pk+j(p-1)+p-k = (k+j)(p-1)+p$. So, modulo $(x,y)^{p^2+1}$, the terms in the sum vanish when $k+j > p$.  Therefore, we can restrict the upper bound on the inner sum to $p-j$ to obtain}
		&= \sum_{j=0}^p \sum_{k=0}^{p-j} \frac{(-1)^j}{j+1} \binom{p(j+1)}{j} \binom{j(p-1)+p}{k} L_1^{k+j+1} (x^p + y^p)^k (x+y)^{j(p-1)+p-k} \\
		&= \sum_{j=0}^p \sum_{i=j}^p \frac{(-1)^j}{j+1} \binom{p(j+1)}{j} \binom{j(p-1)+p}{i-j} L_1^{i+1} (x^p + y^p)^{i-j} (x+y)^{p+jp-i} \\
		&= \sum_{i=0}^p \sum_{j=0}^i \frac{(-1)^j}{j+1} \binom{p(j+1)}{j} \binom{j(p-1)+p}{i-j} L_1^{i+1} (x^p + y^p)^{i-j} (x+y)^{p+jp-i}.
	\end{align*}
	Here, we have set $i=k+j$. For the final step, note that the second sum is over $i,j$ such that $0 \leq j \leq i \leq p$.  This condition is equivalent to the condition that $0 \leq i \leq p$ and $ 0 \leq j \leq i$; hence, the sums are the same.
	
	Evaluating the final term, $L_2 x^{p^2}$ at $\log(x) + \log(y)$ we have
	\begin{align*}
		L_2 (\log(x) + \log(y))^{p^2} &= L_2 (x+y)^{p^2} \mod (x,y)^{p^2+1}
	\end{align*}
	So, modulo $(x,y)^{p^2+1}$, and substituting for $L_1$ and $L_2$ using \eqref{eq:Ls}, 
	we have
	\begin{align*}
		\exp&(\log(x) + \log(y)) \\
		&= (x+y) + L_1 (x^p + y^p) + L_2(x^{p^2} + y^{p^2}) \\
		&\qquad -L_1(x+y)^p - \sum_{i=1}^p L_1^{i+1}\sum_{j=0}^i \frac{(-1)^j}{j+1} \binom{p(j+1)}{j} \binom{j(p-1)+p}{i-j} (x^p + y^p)^{i-j} (x+y)^{p+pj - i}\\
				&= x+y-\frac{u_1}{p-p^{p}} pC_p (x,y) \\
		&\qquad - \sum_{i=1}^p u_1^{i+1} \frac{1}{(p-p^p)^{i+1}} \sum_{j=0}^i \frac{(-1)^j}{j+1} \binom{p(j+1)}{j} \binom{j(p-1)+p}{i-j} (x^p + y^p)^{i-j} (x+y)^{p+pj-i} \\
		&\qquad - \frac{1}{p-p^{p^2}} \left( 1 + \frac{u_1^{p+1}}{p-p^p} \right)pC_{p^2}(x,y) \\
		&= x+y - \frac{u_1}{1-p^{p-1}} C_p (x,y) - \sum_{i=1}^p u_1^{i+1} P_{p+i(p-1)} (x,y) - \frac{1}{1-p^{p^2-1}} \left( 1 + \frac{u_1^{p+1}}{p-p^p} \right)C_{p^2}(x,y). 
	\end{align*}
	\end{proof}

\subsection{The action}
The following is Proposition 3.2 in \cite{lader}. To make sense of the statement, recall the following facts. Fix $g= \sum_{i\geq 0} g_i S^i$ in $\SS$ with $g_i^{p^2}-g_i=0$. Let $g_* \colon R \to R$ be ring homomorphism given by the left action of $\SS$ on $R$.
Then there is an associated $\star$-isomorphism $h_g \colon g_*F \to F$, and since $F$ is $p$-typical, it takes the form
\begin{equation}\label{eq:hg} h_g(x) =\sideset{}{^F} \sum_{i\geq 0}t_i(g) x^{p^i}\end{equation}
for $t_i(g) \in R$ such that $t_i(g) = g_i$ modulo $(p,u_1)$. In particular, $t_0$ is a unit. Further, note that $[p]_F(x) = px +_F u_1 x^p +_F x^{p^2}$
and the $t_i(g)$ satisfy the following recursive formula
\begin{equation}\label{eq:recursive}
h_g([p]_{g_*F}(x)) = [p]_{F}(h_g(x)).\end{equation}
Below, we fix $g$ and abbreviate $t_i = t_i (g)$.

\begin{theorem}[Lader]\label{thm:tislad}
 Let $p$ be any prime. Let $g \in \SS$. Then
\begin{enumerate}[(a)]
	\item $g_* (u_1)= t_0^{p-1}u_1 + t_0^{-1}t_1(p-p^p)$,
	\item $t_0 = t_0^{p^2}+ u_1 t_1^p  -  t_0^{p(p-1)} t_1 u_1^p $ modulo $(p)$, and 
	\item $t_1 = t_1^{p^2} + t_2^p u_1 - \sum_{i=1}^{p-1} \frac{1}{p} \binom{p}{i} u_1^{i+1} t_1^{pi} t_0^{p^2 (p-i)}$ modulo $(p, u_1^{p+1})$
\end{enumerate}
\end{theorem}
\begin{proof} 
First, studying \eqref{eq:recursive} modulo $(x^{p+1})$ gives
\[t_0( px\underset{g_*F}+g_* (u_1) x^p  ) \underset{F}+  t_1( px )^p = p( t_0 x +_F t_1 x^p) +_F u_1(t_0x)^p. \]
The higher order terms are all of order greater than $x^p$, so this reduces to
\[ t_0px+t_0g_*(u_1)x^p + t_1p^p x^p = pt_0x+ pt_1x^p + u_1 t_0^p x^p.\]
Comparing the coefficients of $x^p$ gives (a).

Using this result modulo $(p)$, \eqref{eq:recursive} gives the following equality:
	\begin{equation}\label{eq:theone}	\sideset{}{^F} \sum_{i \geq 0} t_i \left(t_0^{p-1} u_1 x^p \underset{g_*F}+ x^{p^2} \right)^{p^i} = u_1 \left( \sideset{}{^F} \sum_{i \geq 0} t_i x^{p^i} \right)^p \underset{F}+ \left( \sideset{}{^F} \sum_{i \geq 0} t_i x^{p^i} \right)^{p^2}. \end{equation}
This trivially reduces to
	\[
		t_0 \left(t_0^{p-1} u_1  x^p \underset{g_* F}+ x^{p^2} \right) \underset{F}+ t_1 \left( t_0^{p-1} u_1 x^p  \underset{g_* F}+ x^{p^2} \right)^p = u_1 \left( t_0 x \underset{F}+ t_1 x^p \underset{F}+ t_2 x^{p^2} \right)^p \underset{F}+ \left( t_0 x \underset{F}+ t_1 x^p \underset{F}+ t_2 x^{p^2} \right)^{p^2} .
	\]
The higher order terms are divisible by $x^{p^2+1}$, so we conclude that
	\[
		 t_0^{p} u_1  x^p + \big( t_0 + t_1 (t_0^{p-1} u_1 )^p \big) x^{p^2} = u_1 t_0^p x^p + \big( u_1 t_1^p + t_0^{p^2}\big) x^{p^2}. 
	\]
Part (b) follows by comparing coefficients of $x^{p^2}$.

	The proof of part (c) is more involved: it is proved by comparing the coefficients of $x^{p^3}$.  First, using parts (a) and (b) we compute the following modulo $(x^{p^3+1}, u_1^{p+1})$, using the fact that $g_* (u_1) x^p \underset{g_* F}+ x^{p^2} =  t_0^{p-1} u_1 x^{p} + x^{p^2} $ modulo $(x^{p^{2}+p})$:
	\begin{align*}
		\sideset{}{^F} \sum_{i \geq 0} t_i \left( g_* (u_1) x^p \underset{g_* F}+ x^{p^2} \right)^{p^i}  
			&=
	 t_0 \left( t_0^{p-1} u_1 x^p \underset{g_* F}+ x^{p^2} \right) \underset{F}+ t_1\left( t_0^{p-1} u_1 x^{p} + x^{p^2} \right)^p \underset{F}+ t_2 \left( t_0^{p-1} u_1 x^p \right)^{p^2}  \\
		&= t_0 \left( t_0^{p-1} u_1 x^p \underset{g_* F}+ x^{p^2} \right) \underset{F}+ \left( t_1 t_0^{p(p-1)} u_1^p x^{p^2} + t_1 x^{p^3} \right) 
	\end{align*}
	Therefore, the coefficient of $x^{p^3}$ is $c + t_1$ where $c$ is the coefficient of $x^{p^3}$ in 
		\[	t_0 \left( t_0^{p-1} u_1 x^p \underset{g_* F}+ x^{p^2} \right) \underset{F}+ t_1 t_0^{p(p-1)} u_1^p x^{p^2}. \]
	To compute this coefficient define 
		\[	X = t_0^{p-1} u_1 x^p, \qquad Y = x^{p^2}, \qquad Z = t_1 t_0^{p(p-1)} u_1^p x^{p^2}. \]
	First, we note that $X^i Z^j = 0$ modulo $(u^{p+1})$ for $i,j > 0$.  Then, letting 
	\[F(s,t) = s+t +\sum_{i,j>0} a_{i,j}s^i t^j,\] 
	we have that $c$ is the coefficient of $x^{p^3}$ in the following expression which is computed modulo $(u_1^{p+1}, x^{p^3+1})$:
		\begin{align*}	
			t_0 (X \underset{g_* F}+ Y) \underset{F}+ Z &= t_0 (X \underset{g_* F}+ Y) + Z + \sum_{i,j > 0} a_{i,j} (t_0 X \underset{g_* F}+ Y)^i Z^j \\
			&= t_0 (X \underset{g_* F}+ Y) + Y + Z + \sum_{i,j > 0} a_{i,j} Y^i Z^j - Y  \\
			&= t_0 (X \underset{g_* F}+ Y) + Y \underset{F}+ Z - Y .
		\end{align*}
		From \fullref{thm:Fform}, we have that, modulo $(x^{p^3+1})$,
		\[
			Y \underset{F}+ Z =  Y + Z  - \frac{u_1}{1-p^{p-1}}C_p (Y,Z) -  \sum_{i=1}^p u_1^{i+1} P_{p+i(p-1)}(Y,Z) 
		\]
		We have dropped the term involving $C_{p^2}(Y,Z)$ since it has degree greater than ${p^3}$.
		Each monomial in $C_p(Y,Z)$ is a multiple of $Z$, so $u_1C_p(Y,Z)$ vanishes modulo $(u_1^{p+1})$. The terms of the sum indexed by $i$ are homogeneous of degree $p^3+i(p^3-p^2)$ and so vanish modulo $(x^{p^3+1})$. 
		Therefore, the coefficient of $x^{p^3}$ in $Y \underset{F}+ Z$ is zero modulo $(u_1^{p+1})$.

	Using \fullref{thm:Fform} again, we have that modulo $(x^{p^3+1})$
		\[
			X \underset{g_* G}+ Y = X + Y - \frac{u_1}{1-p^{p-1}} C_p (X,Y) - \sum_{i=1}^p u_1^{i+1} P_{p + i(p-1)} (X,Y)
		\]
		Again, the term involving $C_{p^2}(X,Y)$ has degree greater than $p^3$ and has been omitted.
The highest power of $x$ in $C_p (X,Y)$ is $p^3-p^2+p$, so this term in the sum cannot contribute to the coefficient of $x^{p^3}$.  
This leaves the sum indexed by $i$. Fix $i$. Then if $i-j\geq1$, the monomials of $P_{p+i(p-1)}$ are zero modulo $(u_1^{p+1}, x^{p^3+1})$. So we only consider the terms such that $i=j$, which gives
		\begin{align*}
			\sum_{i=1}^p &u_1^{i+1} \frac{1}{(p-p^p)^{i+1}} \frac{(-1)^i}{i+1} \binom{p(i+1)}{i} (X+Y)^{p+i(p-1)}.
		\end{align*}		
When the power of $Y=x^{p^2}$ in the binomial expansion $(X+Y)^{p+i(p-1)}$ exceeds $p$, the monomials vanish modulo $(x^{p^3+1})$. In the remaining monomials, the exponent of $X =  t_0^{p-1} u_1 x^p$ is at least $i(p-1)$, so that after multiplying with 
$u_1^{i+1}$, these monomials vanish modulo $(u_1^{p+1})$. 
Therefore, the coefficient of $x^{p^3}$ in $X +_F Y$ is also zero modulo $(u_1^{p+1})$. We conclude that $c=0$ modulo $(u_1^{p+1})$ which implies that the coefficient of $x^{p^3}$ on the left hand side of \eqref{eq:theone} is $t_1$ modulo $(u_1^{p+1})$.
	
	Now, we need to compute the coefficient of $x^{p^3}$ on the right hand side of \eqref{eq:theone}. Modulo $(u_1^{p+1}, x^{p^3+1})$, we have,
	\begin{align*}
		u_1 \left( \sideset{}{^F} \sum_{i \geq 0} t_i x^{p^i} \right)^p \underset{F}+ \left( \sideset{}{^F} \sum_{i \geq 0} t_i x^{p^i} \right)^{p^2} &= u_1 ( t_0x \underset{F}+ t_1 x^p \underset{F}+ t_2 x^{p^2})^p \underset{F}+ (t_0x \underset{F}+ t_1 x^p \underset{F}+ t_2 x^{p^2})^{p^2} \\
		&= u_1 \big((t_0 x \underset{F}+ t_1x^p)^p + t_2^p x^{p^3} \big) \underset{F}+ (t_0^{p^2}x^{p^2} + t_1^{p^2} x^{p^3}) \\
		&=( u_1 (t_0x \underset{F}+ t_1x^p)^p \underset{F}+ t_0^{p^2} x^{p^2} )+ u_1 t_2^p x^{p^3} + t_1^{p^2} x^{p^3} &
	\end{align*}
	We apply \fullref{thm:Fform}. Using the fact that we are working modulo $(p)$, that $C_{p^2} (t_0 x, t_1 x^p)^p = 0 $ modulo $(x^{p^3+1})$, and that, modulo $( u_1^{p+1}, x^{p^3+1})$,
		\begin{align*}
		u_1 (t_0x \underset{F}+ t_1 x^p)^p 
		&= u_1 \left( t_0 x + t_1 x^p - \frac{u_1}{1-p^{p-1}} C_p (t_0x, t_1 x^p) - \sum_{i=1}^p u_1^{i+1} P_{p+i(p-1)} (t_0 x, t_1 x^p) \right)^p \\
		&= u_1 (t_0^p x^p + t_1^p x^{p^2})
	\end{align*}
the problem reduces to computing the coefficient of $x^{p^3}$ in
		\[	(u_1 t_0^p x^p + u_1 t_1^p x^{p^2}) \underset{F}+ t_0^{p^2} x^{p^2} + \big( u_1 t_2^p + t_1^{p^2} \big) x^{p^3}. \]
	Let
		\[	A = u_1 t_0^p x^p, \qquad B = u_1 t_1^p x^{p^2}, \qquad C = t_0^{p^2} x^{p^2}. \]
	Then the coefficient of $x^{p^3}$ in the preceding expression is $c + (u_1 t_2^p + t_1^{p^2})$ where $c$ is the coefficient of $x^{p^3}$ in
		\[	(A+B) \underset{F}+ C. \]
	Using \fullref{thm:Fform} once again, we have that modulo $(x^{p^3+1})$
		\begin{align*}	
			(A+B) \underset{F}+ C &= A+B+C - \frac{u_1}{1-p^{p-1}} C_p (A+B,C) - \sum_{i=1}^p u_1^{i+1} P_{p+i(p-1)}(A+B,C)
		\end{align*}
	dropping as usual the term involving $C_{p^2}(A+B,C)$ for degree reasons. Since $A+B$ is divisible by $u_1x^p$ and $C$ by $x^{p^2}$, a slightly tedious but straightforward inspection of the sum indexed by $i$ shows that it vanishes modulo $(u_1^{p+1}, x^{p^3+1})$.
Clearly, $A+B+C$ has no powers of $x^{p^3}$, so cannot contribute to the coefficient of $x^{p^3}$. It remains to inspect $ C_p (A+B,C) $. We have
\begin{align*}
 C_p (A+B,C)  &= \sum_{k=1}^{p-1}\frac{1}{p}\binom{p}{k}u_1^{k} (t_0^px^p + t_1^px^{p^2})^{k} t_0^{p^2(p-k)} x^{p^2(p-k)} \\
 &= \sum_{k=1}^{p-1}\frac{1}{p}\binom{p}{k}u_1^{k} \left(\sum_{\ell=0}^{k} \binom{k}{\ell} t_0^{p\ell}  t_1^{p(k-\ell)}x^{p\ell + p^2(k-\ell)}\right) t_0^{p^2(p-k)} x^{p^2(p-k)} \\
  &= \sum_{k=1}^{p-1}\frac{1}{p}\binom{p}{k}u_1^{k} \left(\sum_{\ell=0}^{k} \binom{k}{\ell} t_0^{p\ell+p^2(p-k)}  t_1^{p(k-\ell)}x^{p^3-\ell(p^2-p) }\right)  
\end{align*}
The terms of degree $p^3$ correspond to those for which $\ell=0$, which is exactly $C_p(B, C)$.  
Hence, the coefficient of $x^{p^3}$ in $(A+B) \underset{F}+ C$ is equal to the coefficient of $x^{p^3}$ in 
		\begin{align*}	
			-\frac{u_1}{1-p^{p-1}} C_p (B, C)
			&= -\frac{1}{p-p^p}\left( \sum_{k=1}^{p-1} \binom{p}{k} u_1^{k+1} t_0^{p^2 (p-k)} t_1^{pk} \right) x^{p^3}.
		\end{align*}
	So, combining this with the above, we get that the coefficient of $x^{p^3}$ on the right hand side of \eqref{eq:theone} is
		\[	t_1^{p^2} + t_2^p u_1 - \frac{1}{p-p^p} \sum_{i=1}^{p-1} \binom{p}{i} u_1^{i+1} t_1^{pi} t_0^{p^2 (p-i)}\]
			modulo $(u_1^{p+1})$.
	Hence, equating coefficients, we have
		\[	t_1 = t_1^{p^2} + t_2^p u_1 - \sum_{i=1}^{p-1}  \frac{1}{p} \binom{p}{i} u_1^{i+1} t_1^{pi}  t_0^{p^2 (p-i)}  \ \mod (p,u_1^{p+1}) \]
	as claimed.
	\end{proof}
	
We finish with the following result.
\begin{thm}[Lader]\label{thm:include2}
Let $p$ be any prime. Let $g=1+g_1S+g_2S^2$ modulo $(S^3)$. Then
\[t_0= 1+g_1^pu_1-g_1u_1^p+(g_2-g_2^p)u_1^{p+1}+\sum\limits_{i=1}^{p-1}{\frac{1}{p}}{{p}\choose{i}}g_1^{pi}u_1^{p+1+i} + g_1^2u_1^{2p}+g_1^{p}  u_1^{p^2} \mod (p, u_1^{2p+1}).\]
\end{thm}
\begin{proof}
Using the fact that, modulo $(p,u_1)$, $t_0 = 1$, $t_1 = g_1$, and $t_2= g_2$ and the fact that $g_i^{p^2}=g_i$, it follows from part (c) of \fullref{thm:tislad} that
\[t_1 = g_1 + g_2^p u_1 - u_1^2 \sum_{i=1}^{p-1} \frac{1}{p} \binom{p}{i} u_1^{i-1} g_1^{pi}  \mod (p, u_1^{p+1}).\]
From part (b) of \fullref{thm:tislad}, we also conclude that
\[t_0 = 1+g_1^{p}  u_1 -g_1u_1^p \mod (p,u_1^{p+1}).\]

Now, re-substituting these results into part (b) of \fullref{thm:tislad} and computing modulo $(p, u_1^{2p+1})$, we have
\begin{align*}
t_0 &=  (1+g_1^{p}  u_1 )^{p^2} + u_1 \left( g_1 + g_2^p u_1  \right)^p  -(1+g_1^{p}  u_1  )^{p(p-1)} \left( g_1 + g_2^p u_1 - u_1^2\sum_{i=1}^{p-1} \frac{1}{p} \binom{p}{i} u_1^{i-1} g_1^{pi}\right) u_1^p   \\
&=  1+g_1^{p}  u_1^{p^2} + u_1g_1^p + g_2 u_1^{p+1}   -(u_1^p -g_1  u_1^{2p}  ) \left( g_1 + g_2^p u_1 - u_1^2\sum_{i=1}^{p-1} \frac{1}{p} \binom{p}{i} u_1^{i-1} g_1^{pi}\right)  \\
&=  1+ g_1^p u_1-g_1u_1^p + (g_2-g_2^p) u_1^{p+1}   + \sum_{i=1}^{p-1} \frac{1}{p} \binom{p}{i} u_1^{p+i+1} g_1^{pi} +g_1^2  u_1^{2p}+g_1^{p}  u_1^{p^2}  .
\end{align*}
This proves the claim.
\end{proof}

Note that the last term in \fullref{thm:include2} vanishes modulo $(u_1^{2p+1})$ when $p$ is odd.

\subsection{Formulas for the prime $2$}
To prove our results when $p=2$, we need more information on the action of $g$ than what was determined in \fullref{thm:include2}. We gather the information in this section. We note that the computations in this section are computer assisted, but are consistent with the results of \cite{Paper2}, which study the action of the group of automorphisms of the formal group law of a super-singular curve on an associated Lubin--Tate ring.

First, we get specific about the results of \fullref{thm:tislad} and  \fullref{thm:include2} in the case at hand. 
\begin{cor}\label{cor:iniformp2}
Let $p=2$ and $g\in \SS$. Then
\begin{enumerate}[(a)]
	\item $g_* u_1= t_0 u_1$ modulo $(2)$,
	\item $t_0 = t_0^{4}+ u_1 t_1^2  +  t_0^{2} t_1 u_1^2 $ modulo $(2)$, and 
	\item $t_1 = t_1^{4} + t_2^2 u_1 + u_1^{2} t_1^{2} t_0^{4}$ modulo $(2, u_1^{3})$.
\end{enumerate}
Further, for $g =1+ g_1S +g_2 S^2 +\ldots$, 
\begin{align*}
t_1 & = g_1+g_2^2u_1 + g_1^2 u_1^2  \mod (2, u_1^{3}).\\
t_0 & = 1+g_1^2 u_1+g_1u_1^2+(g_2+g_2^2)u_1^{3}+ g_1^{2}u_1^{4} \mod (2, u_1^{5}).
\end{align*}
\end{cor}

The computation of $F(x,y)=\exp(\log(x)+\log(y))$ modulo $(x,y)^{16}$ using the information provided at the beginning of \fullref{sec:unidefcomp} is not expensive for a computer. It would not be enlightening to include the formula here, but the following computations use it, together with the following fact.
\begin{lem}\label{lem:Fmod16}
If $F(x,y)$ is known modulo $(x,y)^{16}$ and $x^2 | X$ and $x^4 | Y$, then $F(X, Y)$ is determined modulo $(x,y)^{34}$.
\end{lem}
\begin{proof}
The error terms for $F(x,y)$ have the form $xy(x,y)^{14}$. If $X$, $Y$ are as stated, the monomials $XY(X,Y)^{14}$ have degree at least $34$. 
\end{proof}
As before, we collect information from the relation
	\begin{equation}\label{eq:theonep2}	\sideset{}{^F} \sum_{i \geq 0} t_i \left(t_0 u_1 x^2 \underset{g_*F}+ x^{4} \right)^{2^i} = u_1 \left( \sideset{}{^F} \sum_{i \geq 0} t_i x^{2^i} \right)^2 \underset{F}+ \left( \sideset{}{^F} \sum_{i \geq 0} t_i x^{2^i} \right)^{4}. \end{equation}
We will study the coefficients in this equation up to that of $x^{32}$ for elements $g \in \SS$ of the form $g= 1+g_2 S^2$ modulo $(S^3)$. Note that $t_1 = g_1$ modulo $(2,u_1)$ and since $g_1=0$, we have $t_1 = 0$ modulo $(2,u_1)$. We also note that, modulo $(2,u_1)$, $F(x,y)$ is equivalent to the Honda formal group law whose coefficients are in $\FF_2$. So, 
\[F(x,y)^{2} = F(x^2,y^2) \mod (2,u_1).\]

\begin{prop}
Let $g = 1 + g_2 S^2 + g_3 S^3 + g_4 S^4 + \ldots$. Then 
\begin{enumerate}[(a)]
\item $t_3 = g_3+g_4^2 u_1$ modulo $(2,u_1^2)$,
\item $t_2 = g_2+ g_3^2u_1+g_1 u_1^2+(g_4   +g_2^2  +g_2^2 ) u_1^3$ modulo $(2,u_1^4)$
\item $t_1 = g_2^2 u_1+g_3 u_1^3+g_3^2 u_1^5+g_3 u_1^6 + ( g_2+g_2^3 +g_4+g_4^2) u_1^7$ modulo $(2,u_1^8)$
\item $t_0 =1+(g_2 + g_2^2) u_1^3+g_3 u_1^5 +g_3 u_1^8  + (g_4+ g_4^2) u_1^9$ modulo $(2,u_1^{10})$.
\end{enumerate}
\end{prop}
\begin{proof}[Computer Assisted Proof.]
For (a), we compute the coefficients of $x^{32}$ modulo $(2,u_1^2)$ in \eqref{eq:theonep2}. For this, we note using the above observations that \eqref{eq:theonep2} reduces to the following relation modulo $(u_1^2,x^{33})$:
\begin{align*}
t_0(t_0 u_1 x^2+_F x^{4}) &+_F t_1 x^{8}  +_F t_2 x^{16} + t_3  x^{32}  \\
 \ \ \ &= u_1(t_0^2x^2+_F t_2^2x^8+_F t_3^2 x^{16} + t_4^2 x^{32})+_F (t_0x +_F t_2x^4+ t_3 x^8 )^4
\end{align*}
By \fullref{lem:Fmod16}, both sides are determined modulo $(x^{34})$ by $F(x,y)$ modulo $(x,y)^{16}$. A direct computation comparing both sides gives
\[ t_3 = t_3^4+t_4^2 u_1  \mod (2,u_1^2).\]
Since $t_i = g_i $ modulo $(2,u_1)$, we get (a).

To get (b) we compute the coefficients of $x^{16}$ modulo $(2,u_1^4)$ in \eqref{eq:theonep2}. Modulo $(2, u_1^4, x^{17})$, we have
\begin{align*}
t_0(t_0 u_1 x^2 \underset{g_*F}+ x^{4}) &+_F t_1(t_0 u_1 x^2 \underset{g_*F}+ x^{4})^2  + t_2x^{16}  \\
 \ \ \ \  &= u_1 \left(  t_0 x +_F t_1x^2+_F t_2x^4 + t_3 x^8 \right)^2 \underset{F}+ \left( t_0 x +_F t_1x^2+ t_2x^4 \right)^{4}. 
\end{align*}
A direct computation comparing both sides gives the relation
\[t_2 = t_2^4+ t_3^2 u_1 +t_1 t_0^2 u_1^2+t_1^4 t_2^2 u_1^2 +t_0^{16} u_1^3+t_2^2 t_0^8 u_1^3+t_1^6 t_0^4 u_1^3+t_0^4 u_1^3 \mod (2,u_1^4).\]
To get the result, we combine this with the fact that $t_i = g_i$ modulo $(2,u_1)$, with (a) and with \fullref{cor:iniformp2}.

To get (c), we compute the coefficient of $x^{8}$ modulo $(2,u_1^8)$ in \eqref{eq:theonep2}.
Modulo $(2, u_1^8, x^{9})$, we have
\begin{align*}
t_0(t_0 u_1 x^2  \underset{g_*F}+  x^{4}) &+_F t_1(t_0 u_1 x^2  \underset{g_*F}+  x^{4})^2  + t_2 t_0^4 u_1^4 x^8   \\ 
 \ \ \ \  &=  u_1 \left(  t_0 x +_F t_1x^2+ t_2x^4 \right)^2 \underset{F}+ \left( t_0 x + t_1x^2 \right)^{4}. 
\end{align*}
A direct computation comparing both sides gives 
\begin{align}\label{eq:bigrelt1}
t_1 = t_1^4+  t_0^8 u_1^4+t_1 t_0^6 u_1^6+t_0^5 u_1^4+t_1^2 t_0^4 u_1^5+t_2 t_0^4 u_1^4+t_1^2 t_0^4 u_1^2+t_1 t_0^3 u_1^3+t_2^2 u_1 \mod (2,u_1^8).
\end{align}
Now, we do a short recursion. First, we use (a), (b) and \fullref{cor:iniformp2} to compute that
\begin{align*}
t_1 &= g_2^2 u_1+g_3 u_1^3 + g_3^2 u_1^5  \mod (2,u_1^6) \\
t_0 &= 1+ (g_2 + g_2^2) u_1^3 \mod (2,u_1^5) .
\end{align*}
We use this again in part (b) of \fullref{cor:iniformp2} and in \eqref{eq:bigrelt1} to finish the proof.
\end{proof}


\end{document}